\documentclass[11pt]{article}
\usepackage{amsthm, amsfonts,epsf,amsmath,amssymb,tikz,url,listings}
\usepackage{algorithm}
\usepackage{multirow}
\usepackage[noend]{algpseudocode}

\makeatletter
\def\BState{\State\hskip-\ALG@thistlm}
\makeatother

\DeclareMathOperator{\aut}{Aut}
\DeclareMathOperator{\pg}{pg}
\DeclareMathOperator{\comp}{Comp}
\DeclareMathOperator{\scc}{sc}

\newtheorem{theorem}{Theorem}
\newtheorem{prop}{Proposition}
\newtheorem{lemma}{Lemma}

\begin{document}
\sloppy
\title{There is no $(75,32,10,16)$ strongly regular graph} 

\author{
Jernej Azarija $^{a}$
\and
Tilen Marc $^{a}$
}

\date{\today}

\maketitle
\begin{center}
$^a$ Institute of Mathematics, Physics and Mechanics, Ljubljana, Slovenia \\
{\tt jernej.azarija@gmail.com}
\\
{\tt tilen.marc@imfm.si}
\medskip
\end{center}

\begin{abstract}
We show that there is no $(75,32,10,16)$ strongly regular graph. The result is obtained  by a mix of  algebraic and computational  approaches. The main idea is to build large enough induced structure and apply the star complement technique. Our result implies that there is no regular two-graph on 76 vertices and no partial geometry with parameters $\pg(4,7,2)$. In particular, it implies that there is no $(76,35,18,14)$ strongly-regular graph. In order to solve this classification problem we also develop an efficient algorithm for the problem of finding a maximal clique in a graph.
\end{abstract}

\section{Introduction}

A popular notion in algebraic graph theory is the concept of a {\em strongly regular} graph. We say that a $k$-regular graph of order $v$ and diameter $2$ is strongly regular with parameters $(v,k,\lambda, \mu)$ if any pair of non-adjacent vertices has precisely $\lambda$ common neighbors while two non-adjacent vertices share $\mu$ common neighbors. 

A fundamental question about strongly regular graphs is for which parameters does a strongly regular graph exist? For example the notorious question about the existence of a graph of order 3250, girth 5, and diameter 2 (also known as Moore graph) asks for the existence of a strongly regular graph (from now on SRG) with parameters $(3250, 57, 0, 1)$. There are essentially four general ways to rule out certain parameters for SRG's from being realizable. An easy double-counting argument gives us the condition that $$(v-k-1) \mu = k (k - \lambda -1)\,. $$ The second condition requires that the numbers $$\frac{1}{2} \left[(v-1)\pm\frac{2k+(v-1)(\lambda-\mu)}{\sqrt{(\lambda-\mu)^2 + 4(k-\mu)}}\right]\,,$$ are integers, and comes from counting the multiplicities of the eigenvalues of a SRG.

Finally one way to rule out certain SRG is through the so-called Krein and absolute bounds see \cite[pp. 231]{Godsil}. All the described criterion's still leave room for many parameters for which it is not known whether there exists such a SRG. The state of affairs for all possible parameters on up to $1300$ vertices is tracked by Brouwer on his web site \cite{Brower}. It can be seen that on up to $100$ vertices there are essentially $15$ parameters whose classification is still open, the smallest three being $(65,32,15,16), (69,20,7,5)$, and $(75,32,10,16)$. We found the parameters of the last one the most intriguing since the existence of the SRG with $(75,32,10,16)$  is connected to the existence of certain so called two-graphs, and referred in \cite[pp. 263]{Godsil} as one of the oldest open problems in this topic. Moreover, it is also connected to the existence of certain partial geometries.

Given that there is no general technique for deciding whether a certain parameter is realizable, a lot of effort has been put into establishing certain structural results about the missing SRG's. Specifically for a potential SRG $X$ with parameters $(75,32,10,16)$, Haemers and Tonchev \cite{Tonchev} showed in 1996 that the chromatic number of $X$ is at least $6$. Four years later Makhnev showed \cite{Makhnev} that $X$  does not contain a $16$-regular subgraph. Recently Behbahani and Lam \cite{Behbahani} also derived some constraints about the structure of the automorphism group of $X$. Particularly, they showed that if $p$ is a prime dividing $|\aut(X)|$, then $p = 2$ or $p = 3$.

In this paper we use the so called star-complement technique \cite{Star} in order to establish that in fact a SRG with parameters $(75,32,10,16)$ does not exist. Moreover, since existence of such a graph is in one to one correspondence with the question of existence of regular two-graph on 76 vertices \cite[pp. 249]{Godsil}, our result implies that there is no regular two-graph with such order. The result has even further consequences: Define a partial geometry $\pg(s,t,\alpha)$, for $s,t,\alpha\geq 1$, as incidence structure $C=(P,L,I)$ consistent of points $P$, lines $L$, and set $I\subset P\times L$ of incidences (we say that $p\in P$ is incident with $l\in L$ if $(p,l)\in I$), such that the following holds. Each pair of points has at most one line incident with both of them. Each line is incident with $s+1$ points, while each point is incident with $t+1$ lines. Finally, if a point $p$ and a line $l$ are not incident, there are exactly $\alpha$ pairs $(q,m)\in I$, such that $p$ is incident with $m$ and $q$ is incident with $l$. The point graph of $\pg(s,t,\alpha)$ is the graph on its points with two points adjacent if they are incident with a common line. This graph is a strongly regular graph \cite{Brower2}. Up to now the existence of $\pg(4,7,2)$ was unknown, but it was known that the point graph of this geometry must be a SGR with parameters $(75,32,10,16)$. Therefore our result shows that also $\pg(4,7,2)$ does not exist.
As an application of Seidel switching one can also deduce that the non-existence of a $(75,32,10,16)$ SRG implies that there is no $(76,35,18,14)$ SRG~\cite{haemers-2}.

The star-complement technique turned out to be a useful tool for re-proving classification results for some SRG's \cite{Milosevic,Stevanovic} although its direct application fails for SRG's with a large number of vertices or valency. The two main drawbacks being the large search space for induced subgraphs and the problem of computing the clique number of some large graphs.

In this paper we address both problems by presenting an application of the interlacing principle that is very effective at pruning the underlying search space as well as an efficient algorithm for finding the clique number. The algorithm for computing the clique number is crucial since computing the clique number of most of our graphs is computationally infeasible for all state of the art programs.

The paper is organized as follows. In the next section we introduce the star-complement technique and the interlacing criterion that we use. In the following two sections we use the star complement technique in order to first establish the clique number of $X$ and finally show that $X$ does not exist. We finish by presenting the obtained results and discussing some of the computational aspects.

\section{Preliminaries} \label{sec:prelim}

\subsection{Star complements}

The idea behind star-complements revolves around the notion of a so called {\em star-complement graph}. Let $G$ be a simple graph of order $n,$ $A_G$ its adjacency matrix and $r$ one of its eigenvalues with multiplicity $f$. We will say that $r$ is an eigenvalue of $G$ whenever we mean that $r$ is an eigenvalue of $A_G$.

An induced subgraph $H \subseteq G$ is called a {\em star-complement for $G$ and eigenvalue $r$} if it has order $n-f$ and $r$ is not an eigenvalue of $H$. As it turns out \cite{Star}, there is a star-complement for every eigenvalue of $G$. For convenience we record this fact in the following proposition.

\begin{prop}
If $G$ is a graph and $r$ an eigenvalue of $G$, then $G$ has a star complement for $r$.
\end{prop}

Before explaining the role of star-complements, let us mention that one can construct a star-complement for an eigenvalue $r$ by extending an induced subgraph of $G$ that does not contain $r$ as an eigenvalue \cite[Lemma 3]{Milosevic}. More precisely:

\begin{prop} \label{Prop-SCextend}
Let $G$ be a graph with eigenvalue $r$. If $H'$ is an induced subgraph of $G$ that does not contain $r$ as an eigenvalue then there exist a star complement $H$ for $G$ and eigenvalue $r$ so that $H'$ is an induced subgraph of $H$.
\end{prop}

The main motivation of star-complements is that they in some way allow us to reconstruct $G$. The reader can find the precise implications in \cite[pp. 150]{Star}, in this paper we shall formulate the theory to suit the needs of our application. 

Let $H$ be a star-complement for $G$ with eigenvalue $r$ and define the inner product 
$$\langle u, v \rangle = u (rI - A_H)^{-1} v^t\,.$$ 

The {\em compatibility graph} of $H$ and $r$ denoted by $\comp(H,r)$ is the graph with vertex set 
$$V(\comp(H, r)) = \{  u \in \{0,1\}^{n-f} \mid \langle u, u \rangle = r \mbox{ and } \langle u, \overrightarrow{1} \rangle = -1\}\,, $$ 

and adjacency defined as $$u \sim v \iff \langle u, v \rangle \in \{-1,0\}\,.$$ Let us remark that the condition that the inner product of a vertex of $\comp(H,r)$ and the all-ones vector is $-1$ does not hold in general but only if we assume that $G$ is a regular graph, see \cite{Row}.

As it turns out, the problem of constructing $G$ is reduced to the problem of finding cliques in $\comp(H,r)$. Specifically

\begin{prop}
If $r$ is an eigenvalue of $G$ with multiplicity $f$, $H$ a star complement for $G$ and $r$, then  $\comp(H,r)$ has a $f$-clique.
\end{prop}

This already sets the general idea behind the application of the star complement technique. Suppose $G$ is a SRG with parameters $(v,k,\lambda, \mu)$ and $r$ an eigenvalue of $G$ with (large) multiplicity $f$. Suppose that we know that $H'$ is a induced subgraph of $G$ and does not have $r$ as an eigenvalue. If $H'$ is large enough we can compute its compatibility graph and check for $f$-cliques. If the obtained graph does not have such a clique then $G$ does not exist.

In most cases we cannot directly find an induced subgraph $H'$ large enough to be a star complement. In that case, by Proposition \ref{Prop-SCextend}, we can extend $H'$ in all possible ways to obtain candidates for a star-complement of $G$ and $r$. Depending on how large $H'$ is, we may obtain a large set of candidates, and for each such candidate $H$ we need to compute the respective clique number of $\comp(H, r)$. 

The set of all possible candidates for star-complements gets large very quickly and hence we need an efficient pruning method that we explain in the next subsection.

\subsection{Interlacing}

If $n > m$ and  $\lambda_1 \ge \cdots \ge \lambda_n$, $\mu_1 \ge \cdots \ge \mu_m$ are two sequences of real numbers, we say that $\{\mu_i\}_{i=1}^m$ {\em interlaces} $\{\lambda_i\}_{i=1}^n$ if $$\lambda_i \ge \mu_i \ge \lambda_{n-m+i} \quad \mbox{for} \quad i \in \{1,\ldots,m\}\,.$$

The well-known interlacing principle states that the eigenvalues of an induced subgraph of $G$ interlace the eigenvalues of $G$. As it turns out this is not a very effective condition for pruning the induced subgraphs obtained in our application. A much more effective criterion is to used the `partitioned' version of the interlacing principle \cite[Cor. 2.3]{Heamers} which we state as follows. Suppose $\mathcal{V} = (\mathcal{V}_1,\ldots,\mathcal{V}_k)$ is a partition of the vertices of $G$. Let $e(\mathcal{V}_i,\mathcal{V}_j)$ denote the number of edges between the vertices of $\mathcal{V}_i$ and $\mathcal{V}_j$ if $i \ne j$, and the number of edges in the graph induced by $V_i$ otherwise. Consider the $k \times k$ matrix $A_{\mathcal{V}} = (a_{i,j})_{i,j=1}^k$ where 
$$a_{i,j} = \left \{ 
\begin{array}{ll}
    \frac{e(\mathcal{V}_i,\mathcal{V}_j)}{|\mathcal{V}_i|} & \quad \mbox{if} \quad i \ne j \\
    \frac{2e(\mathcal{V}_i)}{|\mathcal{V}_i|} & \quad \mbox{if} \quad  i = j
\end{array}
    \right.\,.
    $$

As it turns out, the eigenvalues of $A_{\mathcal{V}}$ interlace the eigenvalues of $G$. 

\begin{prop}\label{Prop-PartitionedAM}
Let $G$ be a graph and $\mathcal{V}$ a partition of its vertices. Then, the eigenvalues of $G$ are interlaced by the eigenvalues of $A_{\mathcal{V}}$.
\end{prop}

Given that we want to decide whether a graph $H$ with vertex set $\{v_1,\ldots,v_k\}$ is an induced subgraph of $G$ we consider the partition $$\mathcal{V} = (\{v_1\},\ldots, \{v_k\}, V(G) \setminus V(H))\,.$$ Since in our scenario $G$ is regular, we can always compute the number of edges in $V(G) \setminus V(H)$ as well as between the vertices of $V(H)$ and $V(G) \setminus V(H)$. We call the matrix that is obtained from $H$ by using this partition {\em the partitioned adjacency matrix} of $H$.

From Proposition \ref{Prop-PartitionedAM} we have that the eigenvalues of the partitioned matrix of $H$ must interlace the eigenvalues of $G$. This turns out to be quite a sharp pruning condition which we will call the {\em interlacing condition}. In particular if the interlacing condition of a graph $H$ is not satisfied, we shall say that $H$ does not interlace $G$.

Let us remark that this condition appears to be more efficient than the positive-definiteness criterion used in \cite{Degraer}. More precisely, we were able to find graphs such that the pruning condition used in \cite{Degraer} is satisfied but they do not interlace our graph. We were however not able to find graphs satisfying the converse situation. 
However, the drawback of this criterion is that the matrices in question are not symmetric and hence computing their eigenvalues is a much less efficient task. 

\subsection{Approach}\label{sec:appr}

The problem of determining whether a $(v,k,\lambda, \mu)$ SRG graph $G$ exists is thus reduced to the following. Pick an eigenvalue $r$ of $G$ with large multiplicity. Start with a large induced subgraph $H'$ that does not have $r$ as an eigenvalue and must appear as an induced subgraph of $G$. Extend $H'$ to a star-complement of $G$ and $r$ using the described pruning conditions to get rid of invalid graphs. Finally, for all potential star-complements $H$ compute the clique number of $\comp(H,r)$. In practice, $\comp(H,r)$ can be a very large and dense graph and we explain how to compute its clique number in Section \ref{Comp-Aspects}.

Now, let us  describe our approach for the classification of SRG with parameters $(75,32,10,16)$. For the eigenvalue we take $r = 2$ and look for a small list $\mathcal{L}$ with graphs of large order such that at least one member of $\mathcal{L}$ is an induced subgraph of a $(75,32,10,16)$ SRG $X$. 
When the list is obtained, we proceed to show that no graph in $\mathcal{L}$ is an induced subgraph of $X$ as follows. For $H \in \mathcal{L}$ let $\scc(H)$ be a largest induced subgraph of $H$ that does not have $2$ as an eigenvalue and has order at most $19$. Note that $\scc(H)$ may not be unique and in this case we can pick an arbitrary such subgraph. If $|V(\scc(H))| = 19$ then $\scc(H)$ is a star complement for $X$ and we use the theory described above to verify that $\omega( \comp(\scc(H),2)) < 56$, and hence that $H$ is not an induced subgraph of $X$. If $|V(\scc(H))| < 19$ then we extend $\scc(H)$ by adding $19-|V(\scc(H))|$ vertices in all possible ways to obtain (by Proposition \ref{Prop-SCextend}) a list of possible star complements for $X$. Again, we show that none of the obtained star complements has a compatibility graph with a large enough clique. The process of extending an induced subgraph $H$ to a graph of order $19$ is done by inductively introducing  new vertices in all possible ways, and in the end removing all candidates that have 2 as an eigenvalue or do not interlace.  In order to minimize the list of candidate graphs we also make use of the following observation. Suppose that there is a pair of vertices $u,v \in V(H)$ that does not yet have many common neighbors in the induced subgraph - that is $u \sim v$ and $|N(u) \cap N(v)| < \lambda=10$, or $u \not \sim v$ and $|N(u) \cap N(v)| < \mu=16$. Suppose further that for every $S \subset V(H) \setminus \{u,v\}$ all the graphs obtained by adding a new vertex  adjacent to $S \cup \{u,v\}$ that interlace $X$ also do not contain $2$ as an eigenvalue. Let us say that such a pair of vertices is {\em graceful}.

In virtue of Proposition \ref{Prop-SCextend} we can simply use these graphs when building a complete list of star complements of $X$ having $H$ as subgraph. Stating it as a proposition

\begin{prop}
If $u,v$ is a graceful pair for $H$ and $\mathcal{L}$ a list of all graphs obtained by adding a new vertex $x$ to $H$ that is joined to $u,v$ and a subset of $V(H) \setminus \{u,v\}$. Then there exist a star complement $G$ for $X$ such that at least one of the members of $\mathcal{L}$ is an induced subgraph of $G$.
\end{prop}

The described approach is performed by the program {\em extend.c} that we describe later. In particular, it turns out that the above procedure is computationally feasible if the list $\mathcal{L}$ of induced subgraphs does not include graphs that are, when reduced to a subgraph without eigenvalue 2, of order less than 17. In practice, this is almost the same as demanding that for each $G \in \mathcal{L}$ we have $n(G)-k_2(G)\geq 17$, where $n(G)$ is the order of $G$ and $k_2(G)$ is the multiplicity of eigenvalue $2$ in $G$.

\section{Cliques in a SRG with parameters $(75,32,10,16)$}

In what follows let $X$ denotes a possible strongly regular graph with parameters $(75,32,10,16)$. Our main goal is to prove that $X$ does not exits. In order to do so we first establish a structural claim related to its cliques. Notice that the Hoffman bound \cite[pp. 204]{Godsil} implies that $\overline{X}$ has independence number at most $5$ and hence that $X$ has clique number at most $5$. On the other hand, Bondarenko, Prymak, and Radchenko developed a general tool for bounding the number of 4-cliques in a strongly regular graph \cite{bono}. In particular, they have established that a SRG with parameters $(75,32,10,16)$ has at least $783$ 4-cliques. 

In this section we show that in fact $X$ has clique number $5$, more precisely, we show the following result.

\begin{prop} \label{prop:clnum}
If $X$ exists, its clique number is $5$. Moreover, every $4$-clique of $X$ is contained in a $5$-clique.
\end{prop}

In order to prove the result we need to recall a very useful lemma whose proof the reader may find in \cite{bono}. Let $H$ be an induced subgraph of order $m$ of a $(v,k,\lambda, \mu)$ strongly regular graph $G$, and let $(d_0,d_1,\ldots,d_{m-1})$ be a vector such that $d_i$ denotes the number of vertices of $H$ having degree $i$ in $H$. Similarly let $(b_0,\ldots,b_m)$ be a vector where $b_i$ denotes the number of vertices of $G-H$ that have $i$ neighbors in $H$. The next lemma gives a relationship between these numbers.

\begin{lemma} \label{Structural}
With notation as above, the following three equations hold 
\begin{align}
\label{eqn:eqlabel}
\begin{split}
\sum_{i=0}^m b_i &= v-m\,,
\\
\sum_{i=0}^m i b_i &= mk - \sum_{i=0}^{m-1} jd_{j}\,,
\\
\sum_{i=0}^m {i \choose 2} b_i &= {m \choose 2} \mu - \sum_{i=0}^{m-1} {i \choose 2} d_i + \frac{1}{2}(\lambda - \mu)\sum_{i=0}^{m-1}id_i \,.
\end{split}
\end{align}

\end{lemma}

Suppose now that $X$ has a $4$-clique $K_4$ that is not contained in a $5$-clique. Letting $H = K_4$ and applying the above Lemma it can easily be verified that there are $4$ solutions $(b_0,b_1,b_2,b_3)$ (notice that by our assumption $b_4 = 0$) to the above system, namely 

\begin{equation}
(3, 20, 48, 0), (0,29,39,3),(1, 26, 42, 2) \mbox{ and } (2, 23, 45, 1).
\end{equation}
 
 In what follows we analyze all these possibilities, showing that none of these solutions occurs as a configuration in $X$. We split the proof into four sections each dealing with a different value of $(b_0,b_1,b_2,b_3)$. The general idea is to use the structure given by a specific configuration to find a small list of graphs that must be an induced subgraph of $X$. For each possible case we have written simple Sage \cite{sage} programs that build graphs with the established structure and prune them using the interlacing principle we described. Whenever we assert that some induced structure is not possible or say that there is a list of graphs satisfying it, there is a corresponding Sage program that computed this part of the claim. Each such program (with the respective output) is recorded in Table \ref{table:sageprog} and is available online, see \cite{GitHub}.  Throughout the rest of the paper $K_4$ will denote a $4$-clique of $X$. 

In the next section we shall analyze the case when $X$ has clique number 5. Since we will also need a glimpse into this case already in this section, let us remark at this point that if $K_5$ is a $5$-clique of $X$ then every vertex in $V(X) - V(K_5)$ has precisely two neighbors in $K_5$. One way to see this claim is to use the above system of equalities which only gives $(0,0,70,0,0,0)$ as a solution vector. Finally, let us remark that throughout the paper we will use the notation $X[S]$ to denote the subgraph of $X$ induced by the set of vertices $S \subseteq V(X)$.

\subsection{Case $(3, 20, 48, 0)$}

Let us denote with $X_0, X_1, X_2$ the subsets of vertices in $V(X)\setminus V(K_4) $ that have, respectively, 0,1, and 2 neighbors in $K_4$. Moreover, denote the vertices in  $X_0$ by $x_1,x_2,x_3$.

\begin{lemma} Every vertex in $X_2$ has precisely two neighbors in $X_0$ while each two vertices in $X_0$ are not adjacent and have all 16 common neighbors in $X_2$.
\end{lemma}

\begin{proof}\label{lem:first}
Let $x_i\in X_0$. We use an argument that will be repeatedly used in this paper. Since $x_i$ is not adjacent to any of the vertices in $K_4$ it has to have 16 common neighbors (since $X$ is strongly regular with $\mu=16$) with each of its vertices. Thus there are $4\cdot 16$ paths of length 2 from $x_i$ to $K_4$. On the other hand, $x_i$ has 32 neighbors ($X$ is 32-regular) in $X_0 \cup X_1 \cup X_2$. All the neighbors are in fact in $X_2$, for otherwise they could not form $64$ $2$-paths to $K_4$. In particular this implies that the vertices of $X_0$ are not adjacent, proving the last statement of the lemma.

Since for $1\leq i < j \leq 3$ it holds $|N(x_i)\cap N(x_j)|=16$, $|N(x_i)|=32$, and $|X_2| = 48$,  we have
$$48\geq |N(x_1) \cup N(x_2) \cup N(x_3)| = 3\cdot 32- 3\cdot 16 + |N(x_1) \cap N(x_2) \cap N(x_3)|,$$

\noindent by the inclusion-exclusion principle. Thus $N(x_1) \cap N(x_2) \cap N(x_3) = \emptyset$. Therefore every vertex in $X_2$ is adjacent to precisely two vertices in $X_0$.
\end{proof}

For $i \in \{1,2,3\}$ let $X_2^{i,j} \subseteq X_2$ be the graphs induced by $N(x_i) \cap N(x_j)$. By the above lemma, the are of order 16.

\begin{lemma}\label{lem:second}
Each of the graphs $X_{1,2},X_{1,3},X_{2,3}$ is isomorphic to the disjoint union of cycles.
\end{lemma}

\begin{proof}
 Let $v \in X_{1,2}$. We count the number of $2$-paths from $v$ to $K_4$. Since $v$ is adjacent to $2$ vertices of $K_4$, there must be $2\cdot 10 + 2\cdot 16$ such paths, counted as in the previous lemma. Denote with $k$ the number of neighbors of $v$ in $X_2$. By the previous lemma, $v$ is adjacent to $2$ vertices in $X_0$, thus it is adjacent to $32-2-2-k$ vertices in $X_1$. By counting $2$-paths we thus have:
$$2\cdot 10 + 2\cdot 16=2k+1(28-k)+0\cdot 2+2\cdot 3\,.$$
 
Therefore, $v$ has $k=18$ neighbors in $X_2$,  and since it is not adjacent to $x_3$ it must have 16 neighbors in $X_2 - X^{1,2} = N(x_3)$. This implies that $v$ has precisely 2 neighbors in $X_{1,2}$.
\end{proof}

In \cite{Makhnev} an analysis of of subgraphs in $X$ (if existent) was made. For two non-adjacent vertices $x,y$ in $X$ call the subgraph of $X$ induced on the common neighbors of $x$ and $y$ a $\mu$-subgraph. It was shown that if $X$ exists, then it does not have a regular $\mu$-subgraph. Lemma \ref{lem:first} states that each of $X_{1,2},X_{1,3},X_{2,3}$ is a $\mu$ subgraph, while Lemma \ref{lem:second} states that it is regular. This cannot be, thus the case $(3, 20, 48, 0)$ is impossible.

\subsection{Case $(1,26,42,2)$}

Let $X_0,X_1, X_2, X_3$ be the sets of vertices having $0,1,2$, and $3$ neighbors in $K_4$, respectively. In particular, let $x_0 \in X_0$ and $x_1\ne x_2 \in X_3$. In what follows we prove a series of claim describing the structure of a graph with this configuration.

\begin{lemma} \label{lem:tab1}
Vertices $x_1$ and $x_2$ are not adjacent.
\end{lemma}

\begin{proof}
Suppose $x_1 \sim x_2$. There are up to isomorphism only two possible induced graphs on $K_4 \cup \{x_1,x_2\}$. Moreover, if we add the vertex $x_0$ we obtain $6$ candidate graphs for an induced subgraph of $X$. None of them interlaces $X$, which was an easy task to check by computer.
\end{proof}

\begin{lemma}\label{lem:b}
Vertex $x_0$ is adjacent to both vertices in $X_3$. Moreover, it has $2$ neighbors in $X_1$ and $28$ neighbors in $X_2$.
\end{lemma}

\begin{proof}
For the sake of contradiction, suppose $x_0$ is adjacent to $k \in \{0,1\}$ vertices of $X_3$. Let $t$ be the number of neighbors of $x_0$ in $X_1$. By double counting  $2$-paths from $x_0$ to $K_4$ we obtain: $$4\cdot 16 = 3k + t + 2(32-k-t)\,,$$ which gives that $k = t$. Without loss of generality suppose that $x_1$ is not adjacent to $x_0$. By counting the number of $2$-paths in a similar way we obtain that $x_1$ has $8$ neighbors in $X_2$. But by strong regularity, $x_0$ and $x_1$ must have $16$ common neighbors which is not possible since $x_0,x_1$ can share at most $k \leq 1$ common neighbors in $X_1$ and $8$ common neighbors in $X_2$. Hence $x_0$ is adjacent to both $x_1$ and $x_2$ and so $k = t = 2$ and the claim follows.
\end{proof}

In virtue of Lemma \ref{lem:b}, let $x_0',x_0''$ be the vertices in $X_1$ that are adjacent to $x_0$.

\begin{lemma}\label{lem:a}
Vertices $x_1,x_2$ each have $19$ neighbors in $X_1$ and $9$ neighbors in $X_2$. In particular, $12$ or $13$ vertices of $X_1$ are adjacent to both $x_1$ and $x_2$, $6$ or $7$ vertices only to $x_1$, and $6$ or $7$ only to $x_2$.
\end{lemma}

\begin{proof}
The first part of the claim is an easy application of the already used double counting argument. Let now $N_{x_1}$ and $N_{x_2}$ be the neighbors of $x_1$ and $x_2$ in $X_1$, respectively. Since $|N_{x_1} \cup N_{x_2} |\leq 26$ (the size of $X_1$), we must have $|N_{x_1} \cap N_{x_2}| \geq 12$ by the inclusion-exclusion principle. On the other hand, $x_1$ and $x_2$ must have 16 common neighbors. Since $x_0$ is a common neighbor and they have 2 or 3 common neighbors on $K_4$ it follows that $|N_{x_1} \cap N_{x_2}| \leq 13$. The other assertions follow easily.
\end{proof}

\begin{lemma}\label{lem:d}
For $i=1,2$, the vertex $x_i$ is adjacent to at least one of the vertices in $\{x_0',x_0''\}$.
\end{lemma}

\begin{proof}
For $i =1,2$, the vertex $x_i$ has $10$ common neighbors with $x_0$. By Lemma \ref{lem:a}, $x_i$ only has $9$ neighbors in $X_2$, thus it must be adjacent to at least one of $x_0', x_0''$.
\end{proof}

Let $X_2^{-0}$ be the set of vertices in $X_2$ that are not adjacent to $x_0$. By Lemma \ref{lem:b}, $|X_2^{-0}|=14$.

\begin{lemma}
At most one vertex from $X_2^{-0}$ is adjacent to $x_1$, and at most one is adjacent to $x_2$.
\end{lemma}

\begin{proof}
Vertex $x_1$ shares at most 2 common neighbors with $x_0$ in $X_1$ (possibly $x_0'$ or $x_0''$). Thus it must have at least $8$ out of $9$ neighbors in $X_2$ adjacent to $x_0$. By symmetry, the claim holds for $x_2$.
\end{proof}

\begin{lemma}
Each vertex in $X[X_2^{-0}]$ that is not adjacent to any of the vertices in $\{x_1,x_2\}$ has degree $t \leq 2$ and it has $t$ neighbors in $\{x_0',x_0''\}$. Vertices (at most two) in $X_2^{-0}$ that are adjacent to exactly one of $x_1$ and $x_2$ have degree $t-1$ in $X[X_2^{-0}]$ and $t\geq 1$ neighbors in $\{x_0',x_0''\}$. If there exists a vertex in $X_2^{-0}$ that is adjacent to both in $x_1$ and $x_2$, then it is adjacent to both vertices in $\{x_0',x_0''\}$ and has degree $0$ in $X[X_2^{-0}]$.
\end{lemma}

\begin{proof}
Let $v\in X_2^{-0}$. Notice that $v$ must have $16$ common neighbors with $x_0$. First, assume it is not adjacent to $x_1$ or $x_2$. Then their common neighbors can only be in $\{x_0',x_0''\}$, say $t$ of them, and in $X_2\setminus X_2^{-0}$. By double counting 2-paths from $v$ to $K_4$ we obtain that $v$ has 16 neighbors in $X_2$. Thus $t$ of them must be in $X_2^{-0}$.

Second, assume that $v$ is adjacent to exactly one of the $x_1,x_2$. Then it has $16-1-t$ neighbors in $X_2\setminus X_2^{-0}$. On the other hand, by double counting, its degree in $X[X_2]$ is $14$. Thus it's degree in $X[X_2^{-0}]$ is $t-1$.

Finally, if $v$ is adjacent to $x_1$ and $x_2$, it has degree $12$ in $X[X_2]$, thus all this neighbors have to be in $X_2\setminus X_2^{-0}$ and it also has to be adjacent to both vertices in $\{x_0',x_0''\}$. 
\end{proof}

\begin{lemma}
Each of the vertices $x_0',x_0''$ has $15-t$ neighbors in $X_2^{-0}$, where $t\in \{1,2\}$ is the number of its neighbors in $\{x_1,x_2\}$. Moreover, $x_0'$ and $x_0''$ are not adjacent.
\end{lemma}

\begin{proof}
By double counting $2$-paths from $x_0'$ to $K_4$ we have that $x_0$ has $25-2t$ neighbors in $X_2$. Vertices $x_0$ and $x_0'$ have $10$ common neighbors. Let $s$ be equal to $1$ if  $x_0'$ and $x_0''$ are adjacent and $0$ otherwise. $x_0$ and $x_0'$ must have $10-t-s$ common neighbors in $X_2$, thus $x_0'$ has $24-2t-(10-t-s)=15-t+s$ neighbors in $X_2^{-0}$. Similar holds for $x_0''$ and since $|X_2^{-0}|=14$, $x_0'$ and $x_0''$ have more than 10 common neighbors. Thus they are not adjacent and $s=0$. The lemma holds.
\end{proof}

The above lemmas give enough structure to be able to computationally obtain a small list of graphs that interlace $X$ and at least one of them must be induced subgraph of $X$, provided that $X$ contains a $K_4$ with this configuration. See Table \ref{table:sageprog} and \cite{GitHub} for the Sage program.

\begin{prop} \label{prop:case126422}
There are $3597$ graphs of the form $K_4\cup \{x_0,x_0',x_0'',x_1,x_2\}\cup X_{2}^{-0}$ that interlace $X$.
\end{prop}

\subsection{Case $(2,23,45,1)$}

Let $X_0 = \{x_0,x_1\}, X_1, X_2, X_3 = \{x_3\}$ be the  sets of vertices having $0,1,2$, and $3$ neighbors in $K_4$, respectively. Again, we start by proving certain structural claims about this configuration.

\begin{lemma} 
$x_0 \not \sim x_1$.
\end{lemma}

\begin{proof}
 If $x_0 \sim x_1$, then the number of $2$-paths from $x_0$ to $K_4$ is at most $3+2\cdot 30$. But since $x_0$ is not adjacent to any vertex of $K_4$, it should have precisely $4\cdot 16$ $2$-paths to it.
\end{proof}

\begin{lemma} 
$x_0 \sim x_3$ and $x_1 \sim x_3$. 
\end{lemma}

\begin{proof}
  Suppose $x_0$ is not adjacent to $x_3$ and let $N_0,N_1$, respectively, be the sets of neighbors of $x_0,x_1$ in $X_2$. By double counting $2$-paths to $K_4$ from $x_0$ or $x_1$ we have $|N_0| = 32$ 
  while $|N_1|= 32-2t$ where $t \in \{0,1\}$ depending on whether $x_1$ is adjacent to $x_3$ or not. Since all the neighbors of $x_0$ are in $X_2$, we have $|N_0 \cap N_1| = 16$. But this implies $|N_0 \cup N_1| = 48-2t \geq 46$ which is not possible since $X_2$ has size $45$.
\end{proof}

\begin{lemma} 
Vertices $x_0$ and $x_1$ have precisely one neighbor in $X_1$. In particular, the two neighbors are distinct.
\end{lemma}

\begin{proof}
 Let $t$ be the number of neighbors of $x_0$ in $X_1$. By counting $2$-paths from $x_0$ to $K_4$ we have $$ 16 \cdot 4 = 3 + t + 2 \cdot (32-1-t)\,,$$ which gives that $t = 1$, and $x_0$ has $30$ neighbors in $X_2$. Same holds for $x_1$. Let again $N_0,N_1$ be the sets of neighbors of $x_0,x_1$ in $X_2$. Then $|N_0 \cup N_1|\leq 45$, thus $|N_0 \cap N_1|\geq 15$. Since $x_0$ and $x_1$ have a common neighbor $x_3$, $|N_0 \cap N_1|= 15$ and $|N_0 \cup N_1| = 45$. This implies also that $x_0$ and $x_1$ cannot have a common neighbor in $X_1$.
\end{proof}

The last two lemmas now imply that $X_2$ can be partitioned into sets $X_2^0, X_2^{\{0,1\}}, X_2^{1}$ each of size $15$ such that every vertex in $X_2^{i}$ is adjacent to $x_i$ and not adjacent to $x_{1-i}$ for $i = 0,1$ and every vertex in $X_2^{0,1}$ is adjacent to both $x_0$ and $x_1$.
Let us denote the neighbors of $x_0,x_1$ in $X_1$ by $x_0'$ and $x_1'$ respectively.

\begin{lemma}\label{lem:e}
If $x_3$ is adjacent to $x_1'$, then it has 1 neighbor in $X_2^0$, otherwise it has no neighbor in $X_2^0$.
\end{lemma}

\begin{proof}
$x_3$ has $10$ neighbors in $X_2$ by double counting of $2$-paths to $K_4$. On the other hand, it must have $10$ common neighbors with $x_1$. Notice that the common neighbors can only be in $X_2 \cup \{x_1'\}$. If $x_3$ is adjacent to $x_1'$, then it must have $9$ neighbors in $X_2^{0,1} \cup X_2^1$ thus $1$ in $X_2^0$. On the other hand, if $x_3$ is not adjacent to $x_1'$, then it must have all $10$ neighbors in $X_2^{0,1} \cup X_2^1$ and no neighbor in $X_2^0$.
\end{proof}

\begin{lemma}
The graph $X[X_2^0]$ has maximal degree $2$. If $v \in X_2^0$ has degree $2$, then it is not adjacent to $x_3$ but it is with $x_1'$, if it has degree $1$, it is adjacent with either both $x_3$ and $x_1'$ or none of them, and if it has degree $0$, it is adjacent to $x_3$ and not adjacent with $x_1'$.
\end{lemma}

\begin{proof}
Pick a vertex $v\in X_2^0$ and let $s\in \{0,1\}$ indicate if it is adjacent with $x_1'$ and $t\in \{0,1\}$ indicate if it is adjacent with $x_3$. Let $r$ be the number of neighbors of $v$ in $X_2$. We count $2$-paths from $v$ to $K_4$:
$$2\cdot 10 +2 \cdot 16 = 2\cdot 3 + 3t + 2r + (32-2-t-r-1) \,. $$
Thus $r=17-2t$. Vertices $v$ and $x_1$ must have $16$ common neighbors. This implies that the number of neighbors of $v$ in $X_2^{0,1} \cup X_2^1$ is $16-t-s$. Hence we have that $v$ has $(17-2t)-(16-t-s)=1-t+s$ neighbors in $X_2^0$ and hence the lemma follows.
\end{proof}

By generating graphs with the established structure (see Table \ref{table:sageprog} and \cite{GitHub}) we infer that none of them interlaces $X$.

\begin{prop}\label{prop:case223451}
No graph of the form  $X_2^0\cup \{x_0,x_1,x_1',x_3\} \cup K_4$ interlaces $X$.
\end{prop}

\subsection{Case $(0,29,39,3)$}

Let $X_1,X_2$ and $X_3 = \{x_0,x_1,x_2\}$ be the respective subsets of vertices of $X$. There are three non-isomorphic ways to introduce $3$ vertices to $K_4$ by joining each vertex to 3 vertex of $K_4$. Each such graph $G_1,G_2,G_3$ can be uniquely described by a tuple $\vec{n} = (n_1,n_2,n_3,n_4)$ counting the number of edges from the $i'$th vertex of $K_4$ to $X_3$. By relabeling the vertices of $K_4$ if needed we obtain three tuples $(0,3,3,3), (1,3,3,2)$ and $(2,2,2,3)$. We proceed by establishing certain structural claims about this configuration, for the simple implementation see Table \ref{table:sageprog} and \cite{GitHub}.

\begin{lemma}  \label{lem:tab9}
The vertices of $X_3$ form an independent set.
\end{lemma}

\begin{proof} 
No matter how we introduce edges among the vertices of $X_3$ in the graph $K_4 \cup X_3$ we do not obtain a graph interlacing $X$. 
\end{proof}

\begin{lemma}\label{lem:aa} Every vertex $x\in \{x_0,x_1,x_2\}$ has $21$ neighbors in $X_1$ and $8$ neighbors in $X_2$. 
\end{lemma}

\begin{proof} Let $x$ have $k$ neighbors in $X_1$ and $l=32-3-k$ neighbors in $X_2$. We count the number of paths of length $2$ from $x$ to $K_4$. We have $3\cdot 10+16=3\cdot 3+k+2 \cdot (32-k-3)$ and thus $k=21$ and $l=8$.
\end{proof}

Let $X_2^0,X_2^1, X_2^3$ be the neighbors in $X_2$ of $x_0,x_1,x_2$ respectively. We have proved that  $|X_2^0| = |X_2^1| = |X_2^2| = 8$. Notice that the sets $X_2^0,X_2^1,X_2^2$ need not be disjoint.

\begin{lemma}\label{lem:ab} It holds $ |X_2^0 \cap X_2^1 |, | X_2^0 \cap X_2^2|,| X_2^1 \cap X_2^2|  \in \{0,1\}$.
\end{lemma}
\begin{proof}
Vertices $x_0$ and $x_1$ have $16$ common neighbors, at least $2$ of them are on $K_4$. Each of $x_0,x_1$ has $21$ neighbors in $X_1$, where $|X_1|=29$. Thus they must have at least $13$ common neighbors in $X_1$. The latter implies that they have at most one common neighbor in $X_2$. Same holds for all other pairs from the assertion of the lemma.
\end{proof}

\begin{lemma}
For $i \in \{1,2\}$ the graph $X[X_2^0 \setminus X_2^i]$ is triangle-free.
\end{lemma}

\begin{proof}
Assume that there exists a triangle in $X[X_2^0 \setminus X_2^i]$. Together with $x_0$ it forms a $K_4$. Assume this $K_4$ does not extend to a $K_5$. Vertex $x_i$ is not adjacent to any of the vertices of this $K_4$. Thus we have have a $4$-clique with some vertices that are not adjacent to it. We have already covered this configuration before and shown that it is not feasible. On the other hand, if $K_4$ extends to a $K_5$, we have a $K_5$ and a vertex that is adjacent to at most one vertex on it. Again, this is not possible, since every vertex not on $K_5$ must have precisely two neighbors on $K_5$. 
\end{proof}

Let $X_1^{-0}$ denote the subgraph of $X_1$ induced on all the vertices not adjacent to $x_0$ and let $X_1^{i}$ be the set of vertices in $X_1$ adjacent to $x_i$ for $i\in \{0,1,2\}$.

\begin{lemma}\label{lem:ac} 
The graph $X_1^{-0}$ has 8 vertices. At most one of the vertices in $X_1^{-0}$ is not adjacent to $x_1$ and at most one is not adjacent to $x_2$. Moreover, the graph on $X_1^{-0}\setminus X_1^i$, for $i\in \{1,2\}$, has no triangles. Each vertex $v\in X_1^{-0}$ has degree $k$ in $X_1^{-0}$ at most 3 and is adjacent to precisely $9-t-m+k$ vertices in $X_2^0$, where $t\in \{0,1,2\}$ is the number of vertices adjacent to $v$ in $\{x_1,x_2\}$ and $m\in \{0,1\}$ the number of common neighbors of $v$ and $x_0$ on $K_4$.  
\end{lemma}

\begin{proof}
By Lemma \ref{lem:aa}, $|X_1^0|=21$, thus $|X_1^{-0}|=8$. For $i\in \{1,2\}$, $x_i$ and $x_0$ share at least 2 neighbors on $K_4$. Since they are non-adjacent, they share 16 neighbors, thus at most 14 in $X_1$. This implies $|X_1^0\cap X_1^i|\leq 14$, thus  $|X_1^0\cup X_1^i| \geq 21+21-14=28$. Since $|X_1|=29$ we see that there exist at most one vertex in $X_1$ that is not adjacent to $x_0$ and to $x_i$.

If there is a triangle in the graph induced by $X_1^{-0}\setminus X_1^i$, this triangle forms a $K_4$ with $x_0$, while $x_i$ is not adjacent to any of its vertices. If this $4$-clique is not a part of a $5$-clique the assertion follows since this case has already been dealt with (a $K_4$ with some vertices not adjacent to it). On the other hand, if this $K_4$ is a part of a $K_5$, we have an induced subgraph of $K_5$ together with a vertex that is adjacent to one or none of the vertices on $K_5$. Again, this is not possible since every vertex not on $K_5$ has precisely two neighbors in $V(K_5)$. Hence $X_1^{0,2}$ is indeed triangle-free. 

Let now $v\in X_1^{-0}$ be as in the lemma. Denote with $j$ the number of its neighbors in $X_1$. By counting 2-paths to $K_4$ we get
$$10+3\cdot 16= 3+3t+j+2(32-1-t-j),$$
hence $j=7+t$. Denote with $l$ the number of neighbors of $v$ in $X_2^0$. Vertex $v$ and $x_0$ have 16 common neighbors, thus
$$16=(j-k)+l+m=(7+t-k)+l+m,$$
from which we get $9-t-m+k=l\leq 8$. This implies also that $k\leq 2$.
\end{proof}

\begin{lemma}\label{lem:notiang} 
Let $v$ be a vertex of $X[X_2^0]$ with degree $k$,  $t\in\{0,1,2\}$ neighbors in $\{x_1,x_2\}$, and $m\in\{1,2\}$ the number of common neighbors of $v$ and $x_0$ on $K_4$. Then:
$$k+m+t\leq 3.$$
In particular $k\leq 2$.
\end{lemma}

\begin{proof}
Let $v\in X_2^0$. Denote  with $j$ the number of neighbors of $v$ in $X_1$. By counting 2-paths from $v$ to $K_4$ we get:
$$2\cdot 10+2\cdot 16 = 2\cdot 3+1\cdot 3+ 3t +j +2(32-2-1-t-j),$$
thus $j=15+t$. Let now $l\leq 8$ be the number of neighbors of $v$ in $X_1^{-0}$. Vertices $v$ and $x_0$ have 10 common neighbors:
$$10=k+m+(j-l)=k+m+(15+t-l),$$
thus $5+k+m+t=l\leq 8$ and $k+m+t\leq 3$. Since $m\in\{1,2\}$, $k\leq 2$.
\end{proof}

By generating all graphs induced by $K_4\cup \{x_0,x_1,x_2\} \cup X_2^0 \cup X_1^{-0}$ we obtain

\begin{prop} \label{prop:case029313}
There are $18089$ non-isomorphic graphs of the form $K_4\cup \{x_0,x_1,x_2\} \cup X_2^0 \cup X_1^{-0}$ that interlace $X$.
\end{prop}

The case analysis carried in this section resulted in a list of $21686$ non-isomorphic graphs, resulting in $6688644$ star complements and roughly $40000$ compatibility graphs. By verifying that they all have clique number smaller than $56$ we established Proposition \ref{prop:clnum}.

\begin{table}[h]
    \begin{tabular}{| l | l | l |} 
    \hline
    Claim & Program & Output \\ \hline
    Lemma \ref{lem:tab1} &  K4/126422/Claim-1.sage & / \\
    Proposition \ref{prop:case126422} & K4/126422/Case126422.sage & K4/12622/cands126422.g6 \\
    Proposition \ref{prop:case223451} & K4/223451/Case223451.sage & / \\
    Lemma \ref{lem:tab9} & K4/029393/Claim1.sage & / \\
    Proposition \ref{prop:case029313} & K4/029393/generateFinal.sage & K4/029393/cands029393.g6 \\
    Lemma \ref{lem:tab14} & K5/triangles/extendTriangle.sage & K5/triangles/candsTriag.g6 \\
    \hline
    \end{tabular}
    \caption{Sage programs constructing small induced structure. Code is available at \cite{GitHub}} \label{table:sageprog}
\end{table}

\section{Main result}
Let $K_5$ be a $5$-clique of $X$ with vertex set $\{k_1,\ldots,k_5\}$.  The only possible configuration of vertices not in $K_5$ is $(0,0,70,0,0,0)$, therefore every vertex of $X$ that is not in $K_5$ has precisely two neighbors in $K_5$. For $1\leq i<j\leq 5$, let $X_{i,j}$ be vertices in $V(X)\backslash V(K_5)$ that are adjacent to $k_i$ and $k_j$. Since $k_i$ and $k_j$ are adjacent, they must have $10$ common neighbors, 3 of them already on $K_5$. Hence $V(G) \setminus V(K_5)$ is partitioned into $10$ sets of $7$ vertices, namely $X_{0,1},X_{0,2},\ldots, X_{4,5}$. In what follows we establish structural results about these partitions.

\begin{lemma}
For any $1 \leq i < j \leq 5$ the graph $X_{i,j}$ is either $\overline{K_7}$ or $K_3 \cup \overline{K_4}$ or $K_1 \cup K_3 \cup K_3$.
\end{lemma}

\begin{proof}
Assume there exists an edge $e = \{x,y\}$ in the graph $X_{i,j}$. Then the vertices $\{x,y,k_i,k_j\}$ induce a $4$-clique. By the result of the previous section, every 4-clique is contained in a  5-clique. Clearly, the additional vertex must be in $X_{i,j}$. Hence we have proved that every edge $e$ in $X_{i,j}$ is contained in a triangle in $X_{i,j}$. Let $T$ be a triangle in $X_{i,j}$ and $v\in X_{i,j}$ a vertex not on $T$. Since $T \cup \{k_i,k_j\}$ induces a 5-clique, every vertex not on this 5-clique is adjacent to exactly 2 vertices on this clique. Since $v$ is adjacent to $k_i$ and $k_j$, it is not adjacent to $T$ and the lemma follows.
\end{proof}

As it turns out every pair of triangles in distinct partitions $X_{i,j}, X_{k,l}$ induce quite a regular structure.

\begin{lemma}\label{lem:tri1}
Let $1 \leq i  < j \leq 5$, $1 \leq k < l \leq 5$ and let $T, T'$ be two triangles of $X_{i,j}$ and $X_{k,l}$, respectively. Let $c = |\{i,j,k,l\}|$. If $c = 3$, then the edges from $T$ to $T'$ form a perfect matching. If $c = 4$, they form a complement of a perfect matching.
\end{lemma}

\begin{proof}\label{lem:tri2}
First assume $c=3$. Since $T \cup \{k_i,k_j\}$ forms a 5-clique, every vertex of $T'$ is adjacent to exactly 2 vertices in this 5-clique. Since $c=3$, it must be adjacent to exactly one vertex in $T$. Similarly, every vertex of $T$ must be adjacent to exactly one vertex in $T'$. Thus, the edges from $T$ to $T'$ form a perfect matching. The case when $c=4$ is similar.
\end{proof}

Our next lemma shows that not all partitions $X_{i,j}$ contain a triangle. In fact at most 7 do. For a simple implementation used in the lemma see Table \ref{table:sageprog} and \cite{GitHub}.

\begin{lemma} \label{lem:tab14}
There are at least three distinct pairs $\{i,j\},\{k,l\},\{m,n\}$ such that $X_{i,j},X_{k,l}$ and $X_{m,n}$ are independent sets of $X$.
\end{lemma}

\begin{proof}
Using Lemma \ref{lem:tri1} and Proposition \ref{Prop-PartitionedAM} we wrote a Sage program generating all possible graphs on $\{k_1,\ldots,k_5\}$ and 8 triangles, each contained in a different set $X_{i,j}$, for $1\leq i<j \leq 5$. There are 2 non isomorphic ways to chose 8 sets for triangles (among 10 sets) and for each such a way there are non-equivalent ways how to choose (complements of) perfect matchings among triangles.
All the obtained graphs in one configuration do not interlace $X$, while in the other configuration only one non-isomorphic example was found giving rise to 117 compatibility graphs. None of them has a clique of order 56. Thus the lemma follows.
\end{proof}

We are now able to prove our main theorem. The lists of graphs obtained in this part are too large to be hosted online hence they are not included in Table \ref{table:sageprog}. However they can be obtained by a request to the authors.

\begin{theorem}\label{thm:main}
The graph $X$ does not exist.
\end{theorem}

\begin{proof}
By the previous lemma, at least 3 graphs among $X_{i,j}$, for $1\leq i<j \leq 5$, are independent sets. It is an easy check that there are 4 non-isomorphic configurations for the choice of 3 sets among $X_{i,j}$. These are: $(X_{1,2},X_{2,3},X_{4,5}),(X_{1,2},X_{2,3},X_{1,3}), (X_{1,2},X_{2,3}, X_{2,4}), (X_{1,2},X_{2,3},X_{3,4})$. Notice that in all combinations we have sets $X_{1,2},X_{2,3}$.

First we analyze the possible candidates for graphs induced by $\{k_1,\ldots,k_5\} \cup X_{1,2} \cup X_{2,3}$. We do this by generating all bipartite graphs that can represent $X_{1,2} \cup X_{2,3}$. 

This is done using McKay's program {\em genbg}. Adding the vertices $\{k_1,\ldots,k_5\}$ and removing non-interlacing graphs we end up with a list of $654325$ graphs. By computing $\scc(G)$ for every such graph $G$ and extending it to have order $19$ we end up with a list of $361547477$ star complements. By computing their respective compatibility graphs and removing isomorphisms we end up with about $100^6$ graphs. We have verified that none of these graphs has a clique of order $56$ hence implying our assertion.
\end{proof}

\section{Computational aspects}\label{Comp-Aspects}

In this section we briefly describe the computational tools and resources used to produce our result. As described in Section \ref{sec:appr}, our approach required generating a list of candidates for an induced subgraph of $X$, compute their compatibility graphs and check their clique numbers. Most programs were written and tested independently in C and Sage, however most of the computation was performed only by C programs due to their efficiency. 

\subsection{Extending graphs and computing compatibility graphs}

As described in Section \ref{sec:prelim} we generated a list of graphs $\mathcal{L}$ such that if $X$ exists then one of the graphs in $\mathcal{L}$ must be an induced subgraph of $X$. In order to rule out the existence of $X$ we had to obtain star complements for each of the graphs in $\mathcal{L}$ and check the clique number of its compatibility graphs. Some of the graphs in $\mathcal{L}$ already had star complements as induced subgraphs and were easy to handle. However some of the graphs did not, and in this case we had to find maximal induced subgraphs with no 2 as eigenvalue, and extend them to have order $19$. When choosing these subgraphs we tried to maximize the order of the automorphism group while minimizing the number of obtained subgraphs - note that two non isomorphic members of $\mathcal{L}$ may have isomorphic subgraphs. Both lists $\mathcal{L}$ and the one obtained from it are available on the GitHub page.

Let us remark that we have found out that the process of extending graphs is computationally feasible whenever the obtained subgraphs have order $17$. For otherwise we obtained far too many graphs of order $19$.

The task of extending graphs to have order $19$ was done by the already introduced program {\em extend.c} which takes as input a file with graphs given in graph6 string format and for each graph outputs all possible ways to introduce a new vertex to it so that the newly obtained graph interlaces $X$. If the input graph has a graceful pair of vertices then the extensions giving the minimal amount of graphs is written. Again, graphs are written in graph6 format.

After each iteration of {\em extend.c} we used McKay's {\em shortg} \cite{Nauty} program to remove isomorphic graphs from the obtained lists. Extending a graph of order 18 takes roughly 0.5 seconds on a standard desktop machine and the whole computation for the proof of our main took roughly 240 CPU hours.

To compute compatibility graphs we wrote a program that takes as input a list of star complements and for each output writes the graph6 representation of its compatibility graph. The program is called {\em compGraph2graph6.c} and is found on GitHub \cite{GitHub}. We have found that the average compatibility graph gets computed in 0.5 seconds and hence the instances of Theorem \ref{thm:main} were computed in about 5000 CPU hours. Let us note that the program does not output compatibility graphs with order smaller than the clique number sought - in our case $57$.

The computationally most intensive part was computing the clique number of the obtained compatibility graph. This step took roughly $150000$ CPU hours and was carried on a computational grid of $2000$ CPU's.

Both programs use the GNU GSL library for linear algebra routines and make use of the precision guaranteed by their implementation. Finally let us remark that in some of the steps we made use of the GNU parallel program \cite{Parallel}. 


\subsection{Computing the clique number}

While it is in general hard to compute the clique number of a graph, the structure of compatibility graphs makes this task a little easier. As one may suspect by its definition, $\comp$ offers a lot of symmetry which we exploit as follows. First let us denote with $G[N(v)]$ the subgraph of $G$  induced by the neighbors of a vertex $v \in V(G)$.

Suppose we wish to compute the clique number $\omega(G)$ and let $v \in V(G)$. Then either $v$ is contained in a maximal clique of $G$ or is not. In the latter case the maximal clique of $G$ equals the maximal clique in $\comp - v$. In other words 

$$ \omega(G) = \max( \omega(G)-v, \omega(G[N(v)]))+1)\,.$$

Now, the key fact in computing $\omega(\comp)$ is that its automorphism group is fairly large and hence in computing its clique number we can remove the entire orbit $o(v)$ of a vertex $v$. Given that $o(v)$ is fairly large, the obtained graph $\comp - o(v)$ is much smaller. We need not stop here. The key property that is used in the above idea is the fact that if $u,v \in V(G)$ are in the same orbit of $\aut(G)$, then the graphs $G[N(u)]$ and $G[N(v)]$ are isomorphic. For our purposes we can define the {\em extended orbit of a graph} $G$ as a partition of $\tilde{\mathcal{O}}(G)$ of $V(G)$ such that two vertices $u,v$ are in the same part if and only if $G[N(u)] \cong G[N(v)]$. Summarizing the above ideas into pseudo code we designed Algorithm \ref{CliqueAlgorithm}. 

\begin{algorithm}
\caption{Algorithm for computing clique numbers of symmetric graphs} \label{CliqueAlgorithm}
\begin{algorithmic}[1]
\Procedure{cliqueNumber(G,c)}{}
\State $\textit{cl} \gets \textit{0}$
\While {$|V(G)| > c$} 
\State $\tilde{\mathcal{O}} \gets extendedOrbits(G)$
\If {$|\tilde{\mathcal{O}}| = |V(G)|$}
\State break
\EndIf
\State $o \gets \text{some orbit of } \tilde{\mathcal{O}}$
\State $v \gets \text{an element of } o$
\State $cltmp \gets cliqueNumber(G[v]),c)+1$
\If {$cltmp > cl$}
\State $cl \gets cltmp$
\EndIf
\State $G \gets G - o$
\EndWhile
\Return $\max(cl, cliqueNumberBruteforce(G))$
\EndProcedure
\end{algorithmic}
\end{algorithm}

In order to compute the clique number of our compatibility graphs we used a variant of Algorithm \ref{CliqueAlgorithm} which leaves out two major details. Namely the computation of the extended orbits of $G$ and the {\em cliqueNumberBruteforce} routine. For the later, we needed an established program that calculates the clique number of a graph. We have found out that on our instances the program {\em mcqd} \cite{Konc} drastically outperforms the well known clique finding algorithm {\em Cliquer} \cite{Cliquer}.  Hence whenever our input graph is small enough, we simply use {\em mcqd}.  Since we only need to determine whether our graph has a clique of size at least $56$ or not we made use of an additional optimization. Suppose we are trying to decide whether a graph $G$ has a clique of size $k$ and the greedy coloring algorithm shows that we can properly color the vertices of $G$ using less than $k$ colors. Then the clique number of $G$ is smaller than $k$ and we can stop our search. This is the essential idea behind the implementation of {\em mcqd} and we used it to obtain an even more efficient test for compatibility graphs.

The second problem of computing the extended orbits is reduced to the problem of computing the orbits of the automorphism groups and canonical forms of graphs. While the computational complexity of these two problems is not settled, it is well known that in practice both problem offer efficient practical solutions.  For example, it takes Bliss \cite{Bliss} about 5 seconds of CPU time on a standard laptop to compute the full automorphism group of a typical compatibility graph of order $6000$ and density $0.4$.

In order to compute the extended orbits of a graph $G$ we first compute the orbits $\mathcal{O}$ of its automorphism group. Finally for every representative of $\mathcal{O}$ we compute the canonical form of $G[N(v)]$ and join orbits with equal canonical forms. 

A simple implementation of the above algorithm was implemented in Sage and is available on the GitHub repository under the name {\em cliqueNumber.sage}. We used a more efficient {\em C++} implementation that we will present in a subsequent paper. 

Finally, let us remark that both {\em mcqd} and {\em Bliss} were integrated into Sage for the purposes of this paper.

\section{Final remarks}

We have shown that a $(75,32,10,16)$ SRG does not exists by presenting a classification approach based on the star complement technique. The main property that we exploited was the fact that such a SRG has an eigenvalue of high multiplicity, namely $56$ which implies that the star complement graph is small ($19$ vertices). Thus one can avoid the combinatorial explosion of constructing all possible star complements, provided that one can build large enough induced structure for the star complement graph. In our case this was established by building the star complement around a maximal clique of our SRG. Two things were crucial for our approach to work. First was the fact that many of the obtained compatibility graphs were isomorphic thus significantly reducing the number of graphs whose clique number was to be determined. The second crucial part was the fact that compatibility graphs had large automorphism groups thus allowing to exploit their symmetries when computing their clique number. We believe that a similar approach can be used to classify at least one of the following open parameters $(69,20,7,5), (95,40,12,20), (96,45,24,18), (99,42,21,15)$.  

\section*{Acknowledgements}
The authors are indebted to Nejc Trdin for kindly sharing his computational resources, Nathann Cohen for Sage development related tasks, and Sandi Klav\v{z}ar for fruitful discussions. The authors would also like to thank Jan Jona Javor\v{s}ek and Barbara Kra\v{s}ovec for their support in using the computational grid. Finally, the authors would like to thank W.H Haemers for his remark about the non-existence of a $(76,35,18,14)$ SRG.

\bibliographystyle{plain}
\bibliography{biblio}

\end{document}